\documentclass[11pt,a4paper]{article}
\usepackage{epsf,epsfig,amsfonts,amsgen,amsmath,amstext,amsbsy,amsopn,amsthm,lineno}
\usepackage{color}
\usepackage{graphicx}
\newtheorem{theorem}{Theorem}

\newtheorem{lemma}{Lemma}

\theoremstyle{definition}

\newtheorem{claim}{Claim}

\newtheorem{remark}{Remark}

\newtheorem{case}{Case}

\begin{document}

\title{\bf\Large Spectral radius and traceability of connected claw-free graphs}

\date{}

\author{Bo Ning$^a$\footnote{E-mail address: ningbo\_math84@mail.nwpu.edu.cn.
Supported by NSFC (No. 11271300) and the Doctorate Foundation of
Northwestern Polytechnical University (cx201326).}~ and Binlong
Li$^{a,b}$\thanks{Corresponding author. Email address:
libinlong@mail.nwpu.edu.cn. Supported by NSFC (No. 11271300) and the
Doctorate Foundation of Northwestern Polytechnical University
(cx201202), and the project NEXLIZ - CZ.1.07/2.3.00/30.0038, which
is co-financed by the European Social Fund and the state budget of
the Czech Republic.}\\[2mm]
\small $^a$ Department of Applied Mathematics,
\small Northwestern Polytechnical University,\\
\small Xi'an, Shaanxi 710072, P.R.~China\\
\small $^b$ Department of Mathematics, University of West Bohemia,\\
\small Univerzitn\'i 8, 306 14 Plze\v n, Czech Republic\\}
\maketitle
\maketitle

\begin{abstract}
Let $G$ be a connected claw-free graph on $n$ vertices and
$\overline{G}$ be its complement graph. Let $\mu(G)$ be the spectral
radius of $G$. Denote by $N_{n-3,3}$ the graph consisting of
$K_{n-3}$ and three disjoint
pendent edges. In this note we prove that: \\
(1) If $\mu(G)\geq n-4$, then $G$ is traceable unless
$G=N_{n-3,3}$. \\
(2) If $\mu(\overline{G})\leq \mu(\overline{N_{n-3,3}})$ and $n\geq
24$, then $G$ is traceable unless $G=N_{n-3,3}$. \\
Our works are counterparts on claw-free graphs of previous theorems
due to Lu et al., and Fiedler and Nikiforov, respectively.
\end{abstract}



\section{Introduction}

Let $G$ be a graph. The \emph{eigenvalues} of $G$ are the
eigenvalues of the adjacency matrix of $G$. Since the adjacency
matrix of $G$ is real and symmetric, all its eigenvalues are real.
The \emph{spectral radius} of $G$, denoted by $\mu(G)$, is the
spectral radius of its adjacency matrix, i.e., the maximum among the
absolute values of its eigenvalues. By Perron-Frobenius' theorem
(see Theorem 0.3 of \cite{CDS}), $\mu(G)$ is equal to the largest
eigenvalue of $G$.

Let $G$ be a graph. We use $e(G)$ to denote \emph{the number of
edges} of $G$. Let $S\subset V(G)$. We use $G[S]$ to denote the
subgraph of $G$ induced by $S$ and $G-S$ to denote the subgraph of
$G$ induced by $V(G)\backslash S$. For a subgraph $H$ of $G$, we use
$G-H$ instead of $G-V(H)$. For two subgraphs $H,H'$ of $G$, we use
$e_G(H,H')$ (or shortly, $e(H,H')$) to denote the number of edges
with one vertex in $H$ and the other one in $H'$.

By $\overline{G}$ we denote the \emph{complement} of $G$. Let $G_1$
and $G_2$ be two graphs. We denote by $G_1+G_2$ the \emph{disjoint
union} of $G_1$ and $G_2$, and by $G_1\vee G_2$ the \emph{join} of
$G_1$ and $G_2$.

A graph $G$ is \emph{traceable} if it has a Hamilton path, i.e., a
path containing all vertices of $G$; and $G$ is \emph{Hamiltonain}
if it has a Hamilton cycle, i.e., a cycle containing all vertices of
$G$. Note that every Hamiltonian graph is traceable. Hamiltonian
properties of graphs have received much attention from graph
theorists. A fundamental theorem due to Dirac \cite{D} states that
every graph on $n$ vertices is traceable if the degree of every
vertex is at least $(n-1)/2$. Up to now, there also has been some
references on the spectral conditions for Hamilton paths or cycles.
We refer the reader to \cite{BC,FN,Hev,LLT,NG,Z}.

In particular, Fiedler and Nikiforov \cite{FN} gave tight sufficient
conditions for the existence of a Hamilton path in terms of the
spectral radii of a graph and its complement.

\begin{theorem}[Fiedler and Nikiforov \cite{FN}]\label{ThFiNi}
Let $G$ be a graph on $n$ vertices. If $\mu(G)\geq n-2$, then $G$ is
traceable unless $G=K_{n-1}+K_1$.
\end{theorem}

\begin{theorem}[Fiedler and Nikiforov \cite{FN}]\label{ThFiNi'}
Let $G$ be a graph on $n$ vertices. If $\mu(\overline{G})\leq
\sqrt{n-1}$, then $G$ is traceable unless $G=K_{n-1}+K_1$.
\end{theorem}

\begin{remark}\label{r1}
Note that $\mu(K_{n-1}+K_1)=\mu(K_{n-1})=n-2$ and
$\mu(\overline{K_{n-1}+K_1})=\mu(K_{1,n-1})=\sqrt{n-1}$.
\end{remark}

Since the connectedness is necessary for studying traceability of
graphs. Lu, Liu and Tian \cite{LLT} presented a sufficient condition
for a connected graph to be traceable.
\begin{theorem}[Lu, Liu and Tian \cite{LLT}]\label{ThLuLiTi}
Let $G$ be a connected graph of order $n\geq 7$. If $\mu(G)\geq
\sqrt{(n-3)^2+3}$, then $G$ is traceable.
\end{theorem}

Lu et al.'s lower bound of spectral radius was sharpened in
\cite{NG}.

\begin{theorem}[Ning and Ge \cite{NG}]\label{ThNiGe}
Let $G$ be a connected graph on $n\geq 7$ vertices. If $\mu(G)\geq
n-3$, then $G$ is traceable unless $G=K_1\vee (K_{n-3}+2K_1)$.
\end{theorem}

The bipartite graph $K_{1,3}$ is called a \emph{claw}. A graph is
called \emph{claw-free} if it contains no induced subgraph
isomorphic to $K_{1,3}$. Claw-free graphs have been a very popular
field of study, not only in the context of Hamiltonian properties.
One reason is that the very natural class of line graphs turns out
to be a subclass of the class of claw-free graphs. However, not
every claw-free graph is Hamiltonian. There are examples of
3-connected non-Hamiltonian claw-free (even line) graphs, but it is
a long-standing conjecture that all 4-connected claw-free graphs are
Hamiltonian (and then, traceable). It is interesting to note that
the lower bound on the degrees in Dirac's theorem for traceability
was lowered to $(n-2)/3$ by Matthews and Sumner \cite{MS} for
claw-free graphs. For a survey on claw-free graphs, we refer the
reader to Faudree et al. \cite{FFR}.

Motivated by the relationship between Dirac's theorem and
Matthews-Sumner's theorem, in this note we will improve the lower
bound in Theorem \ref{ThLuLiTi} and give an analogue of Theorem
\ref{ThFiNi'} for connected claw-free graphs.

Our main results will be listed as follows. By $N_{n-3,3}$ we denote
the graph consisting of a complete graph $K_{n-3}$ with three
disjoint pendent edges.

\begin{center}
\begin{picture}(120,120)

\put(60,50){\put(0,0){\circle{80}} \thicklines
\put(-10,-5){$K_{n-3}$} \put(0,30){\circle*{4}}
\put(0,60){\circle*{4}} \put(0,30){\line(0,1){30}}
\put(20,20){\circle*{4}} \put(40,40){\circle*{4}}
\put(20,20){\line(1,1){20}} \put(-20,20){\circle*{4}}
\put(-40,40){\circle*{4}} \put(-20,20){\line(-1,1){20}}}

\end{picture}

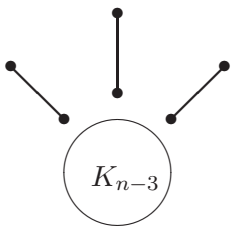
\small Fig. 1. Graph $N_{n-3,3}$.
\end{center}

\begin{theorem}\label{ThMuG}
Let $G$ be a connected claw-free graph on $n$ vertices. If
$\mu(G)\geq n-4$, then $G$ is traceable unless $G=N_{n-3,3}$.
\end{theorem}

\begin{theorem}\label{ThComponent}
Let $G$ be a connected claw-free graph on $n\geq 24$ vertices. If
$\mu(\overline{G})\leq\mu(\overline{N_{n-3,3}})$, then $G$ is
traceable unless $G=N_{n-3,3}$.
\end{theorem}

\section{Preliminaries}

In this section, we first extend the concept of claw-free graphs to
a general one. Let $R$ be a given graph. The graph $G$ is called
\emph{$R$-free} if $G$ contains no induced subgraph isomorphic to
$R$. We will also use three special graphs $L$, $M$ and $N$ (see
Fig. 2). Note that $N=N_{3,3}$.

\begin{center}
\begin{picture}(340,80)
\thicklines

\put(0,20){\put(20,10){\circle*{4}} \put(20,50){\circle*{4}}
\put(60,10){\circle*{4}} \put(60,50){\circle*{4}}
\put(80,10){\circle*{4}} \put(80,50){\circle*{4}}
\put(40,30){\circle*{4}}

\put(20,10){\line(0,1){40}} \put(20,10){\line(1,1){40}}
\put(20,50){\line(1,-1){40}} \put(60,10){\line(0,1){40}}
\put(60,10){\line(1,0){20}} \put(60,50){\line(1,0){20}}

\put(36,36){$a_1$} \put(10,46){$b_1$} \put(10,10){$b'_1$}
\put(62,14){$a_2$} \put(76,14){$b_2$} \put(62,42){$a_3$}
\put(76,42){$b_3$} \put(45,-10){$L$}}

\put(100,20){\multiput(20,30)(20,0){4}{\circle*{4}}
\put(100,10){\circle*{4}} \put(100,50){\circle*{4}}
\put(120,10){\circle*{4}} \put(120,50){\circle*{4}}
\put(40,30){\circle*{4}}

\put(20,30){\line(1,0){60}} \put(80,30){\line(1,1){20}}
\put(80,30){\line(1,-1){20}} \put(100,10){\line(0,1){40}}
\put(100,10){\line(1,0){20}} \put(100,50){\line(1,0){20}}

\put(72,34){$a_1$} \put(54,34){$b_1$} \put(36,34){$c_1$}
\put(18,34){$d_1$} \put(102,14){$a_2$} \put(116,14){$b_2$}
\put(102,42){$a_3$} \put(116,42){$b_3$} \put(65,-10){$M$}}

\put(240,20){\put(20,30){\circle*{4}} \put(60,10){\circle*{4}}
\put(60,50){\circle*{4}} \put(80,10){\circle*{4}}
\put(80,50){\circle*{4}} \put(40,30){\circle*{4}}

\put(20,30){\line(1,0){20}} \put(40,30){\line(1,1){20}}
\put(40,30){\line(1,-1){20}} \put(60,10){\line(0,1){40}}
\put(60,10){\line(1,0){20}} \put(60,50){\line(1,0){20}}

\put(32,34){$a_1$} \put(18,34){$b_1$} \put(62,14){$a_2$}
\put(76,14){$b_2$} \put(62,42){$a_3$} \put(76,42){$b_3$}
\put(45,-10){$N$}}

\end{picture}

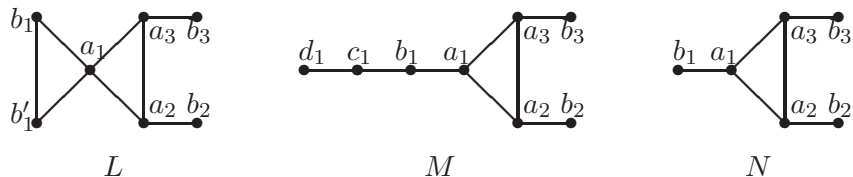
\small Fig. 2. Graphs $L$, $M$ and $N$.
\end{center}

The following two theorems concerning traceability of claw-free
graphs are used in our proofs.

\begin{theorem}[Duffus, Gould and Jacobson \cite{DGJ}]\label{LeDuGoJa}
Every connected claw-free and $N$-free graph is traceable.
\end{theorem}

Adopting the terminology of \cite{F}, we say that a graph is a
\emph{block-chain} if it is nonseparable or it has connectivity 1
and has exactly two end-blocks.

\begin{theorem}[Li, Broersma and Zhang \cite{LBZ}]\label{LeLiBrZh}
Let $G$ be a block-chain. If $G$ is claw-free and $M$-free, then $G$
is traceable.
\end{theorem}

One important tool for studying Hamiltonian properties of claw-free
graphs is the closure theory introduced by Ryj\'{a}\v{c}ek \cite{R}.
It is also useful for our proof. To ensure the completeness of our
text, we include all the terminology and notations as follows. For
other more information, see \cite{R}.

Let $G$ be a graph. Following \cite{R}, for a vertex $x\in V(G)$, if
the neighborhood of $x$ induces a connected but non-complete
subgraph of $G$, then we say that $x$ is \emph{eligible} in $G$. Set
$B_G(x)=\{uv: u,v\in N(x),uv\notin E(G)\}$. The graph $G'_x$,
constructed by $V(G'_x)=V(G)$ and $E(G'_x)=E(G)\cup B_G(x)$, is
called the \emph{local completion of $G$ at $x$}.

As shown in \cite{R}, the \emph{closure} of a claw-free graph $G$,
denoted by $cl(G)$, is defined by a sequence of graphs
$G_1,G_2,\ldots,G_t$, and
vertices $x_1,x_2\ldots,x_{t-1}$ such that \\
(1) $G_1=G$, $G_t=cl(G)$; \\
(2) $x_i$ is an eligible vertex of $G_i$, $G_{i+1}=(G_i)'_{x_i}$,
$1\leq i\leq t-1$; and\\
(3) $cl(G)$ has no eligible vertices.

\begin{theorem}[Ryj\'{a}\v{c}ek \cite{R}]\label{LeRy}
Let $G$ be a claw-free graph. Then $cl(G)$ is also claw-free.
\end{theorem}

\begin{theorem}[Brandt, Favaron and Ryj\'{a}\v{c}ek \cite{BFR}] \label{LeBrFaRy}
Let $G$ be a claw-free graph. Then $G$ is traceable if and only if
$cl(G)$ is traceable.
\end{theorem}

A claw-free graph $G$ is said to be \emph{closed} if $cl(G)=G$. It
is not difficult to see that for every vertex $x$ of a closed graph
$G$, $N_G(x)$ is either a clique, or the disjoint union of two
cliques in $G$ (see \cite{R}). In the following, we say a vertex $x$
of a graph $G$ is a \emph{bad vertex} of $G$ if $N_G(x)$ is neither
a clique, nor the disjoint union of two cliques. So every closed
graph has no bad vertices.

\begin{lemma}\label{LeDegreesum}
Let $G$ be a closed claw-free graph. If there are two nonadjacent
vertices of $G$ have degree sum at least $n-1$, then $G$ is
traceable.
\end{lemma}

\begin{proof}
Let $x,y$ be two nonadjacent vertices of $G$ with degree sum at
least $n-1$. Note that a vertex is nonadjacent to itself. Hence
$x,y$ have at least one common neighbor.

Firstly we assume that $x,y$ have at least three common neighbors,
say $z,z',z''$. Since $G$ is claw-free, either $zz'$ or $zz''$ or
$z'z''$ is in $E(G)$. Without loss of generality, we assume that
$zz'\in E(G)$. Then $z$ is a bad vertex, a contradiction.

Secondly we assume that $x,y$ have two common neighbors, say $z,z'$.
If $zz'\in E(G)$, then $z$ will be a bad vertex. So we have that
$zz'\notin E(G)$. Let $C_x,C'_x,C_y,C'_y$ be the maximal cliques of
$G$ containing $\{x,z\},\{x,z'\},\{y,z\},\{y,z'\}$, respectively.
Clearly $H=G[C_x\cup C'_x\cup C_y\cup C'_y]$ has a Hamilton cycle.
Note that there is at most one vertex in $V(G)\backslash V(H)$.
Since $G$ is connected, we have that $G$ is traceable.

Finally we assume that $x,y$ have only one common neighbor $z$. Then
every vertex is adjacent either to $x$ or to $y$. This implies that
$G$ consists of at most four maximal cliques and $G$ is a
block-chain. Clearly in this case $G$ is traceable.
\end{proof}

The following two lemmas are crucial in the proofs of our two
theorems. We guess that they are of interest in their own rights.

\begin{lemma}\label{LeEdge}
Let $G$ be a connected claw-free graph on $n$ vertices and $m$
edges. If
\begin{align*}
m\geq \binom{n-3}{2}+2,
\end{align*}
then $G$ is traceable unless $G=N_{n-3,3}$ or $L$.
\end{lemma}

\begin{proof}
Let $G'=cl(G)$ be the closure of $G$. Then
\begin{align*}
e(G')\geq m\geq \binom{n-3}{2}+2.
\end{align*}
If $G'$ is $N$-free, then by Theorems \ref{LeDuGoJa} and \ref{LeRy},
$G'$ is traceable, and so is $G$ by Theorem \ref{LeBrFaRy}. Now we
assume that $G'$ contains an induced subgraph $H\sim N$. We denote
the vertices of $H$ as in Fig. 2. In the following part of this
proof, we set $N_H(x)=N_{G'}(x)\cap V(H)$ and $d_H(x)=|N_H(x)|$.

For any $x\in V(G-H)$, note that the neighborhood of $x$ in $G'$ is
either a clique or the disjoint union of two cliques. But any at
least four vertices of $H$ do not form a clique or a disjoint union
of two cliques. This implies that $d_H(x)\leq 3$ for any $x\in
V(G-H)$. Thus
\begin{align*}
e(G')=e(H)+e(G'-H)+e_{G'}(H,G'-H)\leq
6+\binom{n-6}{2}+3(n-6)=\binom{n-3}{2}+3.
\end{align*}
Recall that $e(G')\geq\binom{n-3}{2}+2$. Thus we have
$e(G')=\binom{n-3}{2}+2$ or $\binom{n-3}{2}+3$.

\begin{case}
  $e(G')=\binom{n-3}{2}+3$.
\end{case}

In this case, $G'-H$ is complete and every vertex in $G'-H$ has
exactly three neighbors in $H$. Suppose first that there is a vertex
$x$ in $G'-H$ such that $N_H(x)=\{a_1,a_2,a_3\}$. We claim for every
vertex $x'$ in $G'-H$, $N_H(x')=\{a_1,a_2,a_3\}$. Since
$N_H(x')\neq\{b_1,b_2,b_3\}$, we assume without loss of generality
that $a_1\in N_H(x')$. Note that $xx'\in E(G)$ and $G'[N_{G'}(x)]$
is a clique or disjoint union of two cliques. We can see that
$a_2,a_3\in N_H(x')$. Hence as we claimed $N_H(x')=\{a_1,a_2,a_3\}$.
Thus $G'=N_{n-3,3}$.

Suppose that $E(G')\backslash E(G)\neq\emptyset$. Then
$e(G)=\binom{n-3}{2}+2$ and there is only one edge $e$ in
$E(G')\backslash E(G)$. If $e$ is a pendant edge, then $G$ is
disconnected, a contradiction. So we assume that $e=uv$ is not a
pendant edge. Suppose without loss of generality that $a_1$ is a
vertex in $\{a_1,a_2,a_3\}\backslash\{u,v\}$. Then the subgraph
induced by $\{a_1,b_1,u,v\}$ is a claw in $G$, a contradiction. This
implies that $E(G')\backslash E(G)=\emptyset$. Hence
$G=G'=N_{n-3,3}$.

Now we assume that for every vertex $x\in V(G-H)$,
$N_H(x)\neq\{a_1,a_2,a_3\}$.

If $V(G'-H)=\emptyset$, then $G'=N=N_{3,3}$. By the analysis above,
we can also see that $G=G'=N_{3,3}$. So we assume that
$V(G'-H)\neq\emptyset$.

Let $x$ be a vertex in $G'-H$. Thus $N_H(x)$, and then $N(x)$
induces two disjoint cliques. Note that $N_H(x)\neq\{b_1,b_2,b_3\}$.
We assume without loss of generality that $a_1\in N_H(x)$. If
$a_2\in N_H(x)$, then $a_3\in N_H(x)$; otherwise $a_1$ will be a bad
vertex of $G'$. But in this case $N_H(x)=\{a_1,a_2,a_3\}$, a
contradiction. This implies that $a_2\notin N_H(x)$ and similarly,
$a_3\notin N_H(x)$. Note that $N_H(x)\neq\{a_1,b_2,b_3\}$. We have
$b_1\in N_H(x)$. Without loss of generality, we assume that
$N_H(x)=\{a_1,b_1,b_2\}$. If $G'-H$ has the only one vertex $x$,
then $b_1a_1xb_2a_2a_3b_3$ is a Hamilton path of $G'$. By Theorem
\ref{LeBrFaRy}, $G$ is traceable. Now we assume that there is a
second vertex $x'\in V(G'-H)$.

Since both $\{x,x',b_1,b_2\}$ and $\{x,x',b_2,a_2\}$ induce no
claws, it follows either $a_1,b_1\in N_H(x')$ or $b_2\in N_H(x')$. If
$a_1,b_1\in N_H(x')$, then $b_2\notin N_H(x')$; otherwise $x$ is a
bad vertex of $G'$. Similarly as the case of $x$ above, we can see
that $a_2,a_3\notin N_H(x')$. Thus $N_H(x')=\{a_1,b_1,b_3\}$. If
$b_2\in N_H(x')$, then $a_1,b_1\notin N_H(x')$; otherwise $x$ is a
bad vertex of $G'$. If $a_2\in N_H(x')$, then $b_2$ is a bad vertex
of $G'$, a contradiction. Thus we have $N_H(x')=\{b_2,a_3,b_3\}$. In
conclusion, either $N_H(x')=\{a_1,b_1,b_3\}$ or
$N_H(x')=\{b_2,a_3,b_3\}$.

Suppose that there is a third vertex $x''$. Then similarly as the
case of $x'$, $N_H(x'')=\{a_1,b_1,b_3\}$ or
$N_H(x'')=\{b_2,a_3,b_3\}$. But if $x'$ and $x''$ have the same
neighborhood in $H$, then $x'$ will be a bad vertex, a
contradiction. So we assume without loss of generality that
$N_H(x')=\{a_1,b_1,b_3\}$ and $N_H(x'')=\{b_2,a_3,b_3\}$. Then $x'$
is also a bad vertex, a contradiction. Thus $x,x'$ are the only two
vertices in $G-H$, and $b_1xx'b_3a_3a_1a_2b_2$ is a Hamilton path of
$G'$. By Theorem \ref{LeBrFaRy}, $G$ is traceable.

\begin{case}
$e(G')=\binom{n-3}{2}+2$.
\end{case}

In this case $G=G'$ and there is a vertex $x$ in $G-H$ such that
$d_H(x)=2$ or $xx'\notin E(G)$ for some $x'\in V(G-H)$. Let
$G_1=G-x$. Since every vertex in $G-H-x$ is adjacent to three
vertices in $H$, $G_1$ is connected. Note that
\begin{align*}
  e(G_1)=e(G)-d(x)=\binom{n-3}{2}+2-(n-5)=\binom{n-4}{2}+3.
\end{align*}
Using the conclusion of Case 1, we can obtain that $G_1$ is
traceable or $G_1=N_{n-4,3}$.

Suppose first that $G_1=N_{n-4,3}$. Let $a_1b_1,a_2b_2,a_3b_3$ be
the three pendent edges of $G_1$, where $a_1,a_2,a_3$ are contained
in a clique of $G_1$. Note that $G$ is closed and $N(x)$ is either a
clique or the disjoint union of two cliques. Also note that if $x$
is adjacent to some two vertices of a maximal clique of $G$, then
$x$ will be adjacent to every vertex of the maximal clique of $G$.
Since $d(x)=n-5$, the neighborhood of $x$ does not include
$V(G_1)\backslash\{b_1,b_2,b_3\}$. If $x$ is adjacent to two pendant
vertices, say $b_1,b_2$, then let $P$ be a Hamilton path of the
complete graph $G_1-\{b_1,b_2,b_3\}$ from $a_2$ to $a_3$. Then
$b_1xb_2a_2Pa_3b_3$ is a Hamilton path of $G$. Now we assume that
$x$ is adjacent to exactly one vertex of $\{b_1,b_2,b_3\}$. Suppose
without loss of generality that $b_1\in N(x)$. Since $d(x)=n-5$, we
can see that $n=7$, $G_1=N$ and $N(x)=\{a_1,b_1\}$. Hence $G=L$.

Now we assume that $G_1$ is traceable. Let $P=v_1v_2\ldots v_{n-1}$
be a Hamilton path of $G_1$. If $v_1x\in E(G)$ or $v_{n-1}x\in
E(G)$, then $G$ is traceable. So we assume that $v_1x,v_{n-1}x\notin
E(G)$. If $x$ is adjacent to two successive vertices on $P$, then
$G$ is traceable. So we assume that $x$ is not adjacent to two
successive vertices on $P$. This implies that $n-1-d(x)\geq d(x)+1$.
Since $d(x)=n-5$, we have $n\leq 8$. Note that $n\geq 7$. We can see
that either $xv_2$ or $xv_{n-2}$ is in $E(G)$. We assume without
loss of generality that $xv_2\in E(G)$. Thus $v_1v_3\in E(G)$;
otherwise the subgraph induced by $\{v_2,v_1,v_3,x\}$ is a claw.
Hence $P'=xv_2v_1v_3\ldots v_{n-1}$ is a Hamilton path of $G$.
\end{proof}

\begin{lemma}\label{LeEdge'}
Let $G$ be a connected claw-free graph on $n\geq 24$ vertices and
$m$ edges. If
\begin{align*}
m>{n \choose 2}-\left(1+\sqrt{3n-8}\right)^2,
\end{align*}
then $G$ is traceable unless $G\subseteq N_{n-3,3}$.
\end{lemma}

\begin{proof}
We assume the opposite.

\begin{claim}
$G$ is a block-chain.
\end{claim}

\begin{proof}
Suppose that $G$ is not a block-chain. Since $G$ is claw-free, every
cut-vertex of $G$ is contained in exactly two blocks. This implies
that $G$ has a block $B_0$ which contains at least three
cut-vertices of $G$. Let $a_1,a_2,a_3$ be three cut-vertices of $G$
contained in $B_0$. Let $B_i$, $i=1,2,3$, be the component of
$G-B_0$ which has a neighbor of $a_i$. Let
$H_0=G-(\bigcup_{i=1}^3B_i)$ and $H_i=G[V(B_i)\cup\{a_i\}]$,
$i=1,2,3$. Note that $\nu(H_0)\geq 3$. If
$\nu(H_1)=\nu(H_2)=\nu(H_3)=2$, then $G\subseteq N_{n-3,3}$. Now we
assume without loss of generality that $\nu(H_1)\geq 3$.

Note that $\sum_{i=0}^3\nu(H_i)=n+3$. Thus
\begin{align*}
e(G)=\sum_{i=0}^3e(H_i)\leq\sum_{i=0}^3{\nu(H_i) \choose 2}\leq{n-4
\choose 2}+5\leq{n \choose 2}-\left(1+\sqrt{3n-8}\right)^2
\end{align*}
(noting that $n\geq 24$), a contradiction.
\end{proof}

Let $G'=cl(G)$. If $G'$ is $M$-free, then by Theorems \ref{LeLiBrZh}
and \ref{LeBrFaRy}, $G'$, and then $G$, is traceable. Now we assume
that $G'$ has an induced subgraph $H\sim M$. We denote the vertices
of $H$ as in Fig. 2.

\begin{claim}
Every vertex in $G'-H$ has at most 5 neighbors in $H$; and there is
at most one vertex in $G'-H$ having exactly 5 neighbors in $H$.
\end{claim}

\begin{proof}
Let $x$ be a vertex in $G'-H$. Note that $N_{H}(x)$ is either a
clique or the disjoint union of two cliques. This implies that
$d_{H}(x)\leq 5$. Moreover, if $d_{H}(x)=5$, then
$N_{H}(x)=\{a_1,a_2,a_3,c_1,d_1\}$.

If there are two vertices, say $x$ and $x'$, such that each one has
5 neighbors in $H$, then $N_H(x)=N_H(x')=\{a_1,a_2,a_3,c_1,d_1\}$.
But in this case $x$ will be a bad vertex of $G'$, a contradiction.
\end{proof}

By Claim 2, we have
\begin{align*} e(G)\leq e(G')=e(H)+e(G'-H)+e_{G'}(H,G'-H)\leq
8+{n-8 \choose 2}+4(n-8)+1.
\end{align*}
Thus
\begin{align*}
8+{n-8 \choose 2}+4(n-8)+1>{n \choose
2}-\left(1+\sqrt{3n-8}\right)^2.
\end{align*}
This implies that $n\leq 20$, a contradiction.
\end{proof}
The next theorem we need is a famous theorem due to Hong \cite{Hon}.
In fact, the spectral inequality also works for graphs without
isolated vertices, see \cite{Hon}.

\begin{theorem}[Hong \cite{Hon}]\label{LeHo}
Let $G$ be a connected graph on $n$ vertices and $m$ edges. Then
\begin{align*}
\mu(G)\leq\sqrt{2m-n+1}.
\end{align*}
The equality holds if and only if $G=K_n$ or $K_{1,n-1}$.
\end{theorem}

\begin{theorem}[Hofmeister \cite{Hof}]\label{LeHo'}
Let $G$ be a graph. Then
\begin{align*}
\mu(G)\geq\sqrt{\frac{\sum_{v\in V(G)}d^2(v)}{n}}.
\end{align*}
\end{theorem}

\section{Proofs of the main results}

\noindent{\bf {Proof of Theorem \ref{ThMuG}.}} By Theorem
\ref{LeHo}, $\mu(G)\leq\sqrt{2m-n+1}$. Thus $n-4\leq\sqrt{2m-n+1}$
and
\begin{align*}
m\geq\left\lceil\frac{(n-3)(n-4)+3}{2}\right\rceil=\binom{n-3}{2}+2.
\end{align*}
Note that $\mu(M)=2.6935\ldots<3$. By Lemma \ref{LeEdge}, $G$ is
traceable or $G=N_{n-3,3}$. \qed

\bigskip
\noindent{\bf {Proof of Theorem \ref{ThComponent}.}} We first give a
bound on the value of $\mu(\overline{N_{n-3,3}})$. By using Theorem
2.8 in \cite{CDS} and some computing, we know
\begin{align*}
\mu(K_k\vee (n-k)K_1)=\frac{k-1+\sqrt{4kn-(3k-1)(k+1)}}{2}.
\end{align*}
Thus $\mu(K_3\vee (n-3)K_1)=1+\sqrt{3n-8}$. From the fact
$\overline{N_{n-3,3}}\subset K_3\vee (n-3)K_1$, we obtain
\begin{align*}
\mu(\overline{N_{n-3,3}})<1+\sqrt{3n-8}
\end{align*}
for any $n\geq 6$.

Now we prove the theorem. The idea of our proof comes from
\cite{FN}. We assume that $G$ is not traceable. Let $G'=cl(G)$. By
Theorem \ref{LeBrFaRy}, $G'$ is not traceable. By Lemma
\ref{LeDegreesum}, for any pair of nonadjacent vertices $u,v$ of
$G'$, $d_{G'}(u)+d_{G'}(v)\leq n-2$, and hence
\begin{align*}
d_{\overline{G'}}(u)+d_{\overline{G'}}(v)\geq 2(n-1)-(n-2)=n.
\end{align*}
Furthermore, we have
\begin{align*}
\sum_{v\in V(G)}d_{\overline{G'}}^2(v)=\sum_{uv\in
E(\overline{G'})}(d_{\overline{G'}}(u)+d_{\overline{G'}}(v))\geq
ne(\overline{G'}).
\end{align*}
Note that $\overline{G'}\subseteq\overline{G}$. By Theorem
\ref{LeHo'},
\begin{align*}
\mu(\overline{G})\geq\mu(\overline{G'})\geq\sqrt{\frac{\sum_{v\in
V(G)}d_{\overline{G'}}^2(v)}{n}}\geq\sqrt{e(\overline{G'})}.
\end{align*}
Thus we have
\begin{align*}
e(G')={n \choose 2}-e(\overline{G'})\geq{n \choose
2}-\mu^2(\overline{G})>{n \choose 2}-\left(1+\sqrt{3n-8}\right)^2.
\end{align*}
Recall that $G'$ is claw-free and not traceable. By Lemma
\ref{LeEdge'}, $G'\subseteq N_{n-3,3}$. Thus $G\subseteq N_{n-3,3}$.
But if $G\subset N_{n-3,3}$, then
$\mu(\overline{G})>\mu(\overline{N_{n-3,3}})$, a contradiction. This
implies $G=N_{n-3,3}$. The proof is complete. \qed

\section{Concluding remarks}

In this section, we give a brief discussion of the existence of
Hamilton cycles in claw-free graphs under spectral condition.

Following the notations in \cite{Br}, we use $\mathcal{P}$ to denote
the class of graphs obtained by taking two vertex-disjoint triangles
$a_1a_2a_3a_1$ and $b_1b_2b_3b_1$, and by joining every pair of
vertices $\{a_i,b_i\}$ by a triangle or by a path of order at least
3. We use $P_{x_i,x_2,x_3}$ to denote the graph from $\mathcal{P}$,
where $x_i=T$ if $\{a_i,b_i\}$ is joined by a triangle; and
$x_i=k_i$ if $\{a_i,b_i\}$ is joined by a path of order $k_i\geq 3$.

\begin{center}
\setlength{\unitlength}{0.75pt}\small
\begin{picture}(560,120)
\put(0,0){\multiput(20,30)(80,0){2}{\put(0,0){\circle*{4}}
\put(0,80){\circle*{4}} \put(20,40){\circle*{4}}
\put(0,0){\line(0,1){80}} \put(0,0){\line(1,2){20}}
\put(0,80){\line(1,-2){20}}} \put(20,30){\line(1,0){80}}
\put(20,30){\line(2,1){40}} \put(100,30){\line(-2,1){40}}
\put(20,110){\line(1,0){80}} \put(20,110){\line(2,-1){40}}
\put(100,110){\line(-2,-1){40}} \put(60,50){\circle*{4}}
\put(60,90){\circle*{4}} \put(80,70){\circle*{4}}
\put(60,50){\line(0,1){40}} \put(60,50){\line(1,1){20}}
\put(60,90){\line(1,-1){20}} \put(50,10){$P_{T,T,T}$}}

\put(140,0){\multiput(20,30)(80,0){2}{\put(0,0){\circle*{4}}
\put(0,80){\circle*{4}} \put(20,40){\circle*{4}}
\put(0,0){\line(1,2){20}} \put(0,80){\line(1,-2){20}}}
\put(20,30){\line(1,0){80}} \put(100,30){\line(0,1){80}}
\put(20,30){\line(2,1){40}} \put(100,30){\line(-2,1){40}}
\put(20,110){\line(1,0){80}} \put(20,110){\line(2,-1){40}}
\put(100,110){\line(-2,-1){40}} \put(60,50){\circle*{4}}
\put(60,90){\circle*{4}} \put(80,70){\circle*{4}}
\put(60,50){\line(0,1){40}} \put(60,50){\line(1,1){20}}
\put(60,90){\line(1,-1){20}} \put(50,10){$P_{3,T,T}$}}

\put(280,0){\multiput(20,30)(80,0){2}{\put(0,0){\circle*{4}}
\put(0,80){\circle*{4}} \put(20,40){\circle*{4}}
\put(0,0){\line(1,2){20}} \put(0,80){\line(1,-2){20}}}
\put(20,30){\line(1,0){80}} \put(100,30){\line(0,1){80}}
\put(20,30){\line(2,1){40}} \put(100,30){\line(-2,1){40}}
\put(20,110){\line(1,0){80}} \put(20,110){\line(2,-1){40}}
\put(100,110){\line(-2,-1){40}} \put(60,50){\circle*{4}}
\put(60,90){\circle*{4}} \put(80,70){\circle*{4}}
\put(60,50){\line(1,1){20}} \put(60,90){\line(1,-1){20}}
\put(50,10){$P_{3,3,T}$}}

\put(420,0){\multiput(20,30)(80,0){2}{\put(0,0){\circle*{4}}
\put(0,80){\circle*{4}} \put(20,40){\circle*{4}}
\put(0,0){\line(1,2){20}} \put(0,80){\line(1,-2){20}}}
\put(20,30){\line(1,0){80}} \put(20,30){\line(2,1){40}}
\put(100,30){\line(-2,1){40}} \put(20,110){\line(1,0){80}}
\put(20,110){\line(2,-1){40}} \put(100,110){\line(-2,-1){40}}
\put(60,50){\circle*{4}} \put(60,90){\circle*{4}}
\put(80,70){\circle*{4}} \put(60,50){\line(1,1){20}}
\put(60,90){\line(1,-1){20}} \put(50,10){$P_{3,3,3}$}}
\end{picture}

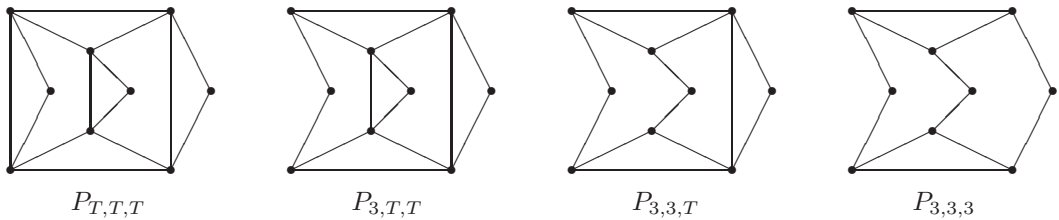
{\small Fig. 3. 2-connected claw-free non-Hamiltonian graphs of
order 9.}
\end{center}

Brousek \cite{Br} showed that every 2-connected claw-free
non-Hamiltonian graph contains a graph in $\mathcal{P}$ as an
induced subgraph. By Brousek's result, we can see that the smallest
2-connected claw-free non-Hamiltonian graphs have order 9, and there
are exactly four such graphs, namely, $P_{T,T,T},$ $P_{3,T,T},$
$P_{3,3,T}$ and $P_{3,3,3}$, see Fig. 3.

Let $H$ be a graph from Fig. 3, and let $G$ be a graph obtained from
$H$ by replacing one triangle by a complete graph $K_{n-6}$. Then
$G$ is not Hamiltonian and $\mu(G)>n-7$. Recently, we get the
following result.
\begin{theorem}\label{ThHamiltonian}
Suppose that $G$ is a 2-connected claw-free graph of sufficiently
large order $n$. If $\mu(G)\geq n-7$, then $G$ is Hamiltonian or $G$
is a subgraph of a graph which is obtained from $P_{T,T,T}$,
$P_{3,T,T}$, $P_{3,3,T}$ or $P_{3,3,3}$ by replacing a triangle by
$K_{n-6}$.
\end{theorem}

For further works on this topic, we refer the reader to \cite{LN}.

\section*{Acknowledgements}
We would like to show our gratitude to two anonymous referees for
their invaluable suggestions which largely improve the quality of
this paper, especially for pointing out an error in our original proof of Theorem \ref{ThComponent}.


\end{document}